\title{String graphs and separators}
\author{
{\sc Ji\v{r}\'{\i} Matou\v{s}ek}\thanks{Supported
by the  ERC Advanced Grant No.~267165
and by the project CE-ITI (GACR P202/12/G061).}
\\
   {\footnotesize Department of Applied Mathematics}\\[-1.5mm]
   {\footnotesize  Charles University, Malostransk\'{e} n\'{a}m. 25}\\[-1.5mm]
{\footnotesize  118~00~~Praha~1,
   Czech Republic, and}\\%[-1.5mm]
{\footnotesize    Institute of  Theoretical Computer Science}\\[-1.5mm]
{\footnotesize    ETH Zurich,
      8092 Zurich, Switzerland}
}
\date{}
\newif\ifafour
\afourtrue

\documentclass[11pt]{article}

\usepackage{a4}

\usepackage{graphicx,url}
\usepackage{color}

\usepackage{amsmath,amsfonts,amssymb,dsfont,theorem}

\newtheorem{theorem}{Theorem}[section]

\newtheorem{lemma}[theorem]{Lemma}

\newtheorem{proposition}[theorem]{Proposition}

\newtheorem{observation}[theorem]{Observation}

\newtheorem{claim}[theorem]{Claim}

\newtheorem{exercise}[theorem]{Exercise}
\newtheorem{problemo}[theorem]{Problem}
\newtheorem{example}[theorem]{Example}
\newtheorem{fact}[theorem]{Fact}
\newcommand{\qed}{\hfill\ensuremath{\Box}}
\newenvironment{proof}{\noindent\textbf{Proof.}
}{\qed\par\medskip}

\newenvironment{proofof}[1]{\medskip\noindent\textbf{Proof of #1.}
}{\qed\par\medskip}
\newenvironment{proofhd}[1]{\noindent\textbf{#1.}}{\qed\par\medskip}
{\ifx&#1&%
  \begin{problemo}\else\begin{problemo}[#1]\fi\upshape}%
    {\end{problemo}}

\newcommand{\N}{\ensuremath{\mathds N}}

\newcommand{\R}{\ensuremath{\mathds R}}

\newcommand{\CC}{\ensuremath{\mathcal C}}
\newcommand{\MM}{\ensuremath{\mathcal M}}
\newcommand{\PP}{\ensuremath{\mathcal P}}

\newcommand\makevec[1]{{\boldsymbol{#1}}}
\def \bb {\makevec{b}}

\def \xx {\makevec{x}}
\def \yy {\makevec{y}}
\def \cc {\makevec{c}}
\def \bb {\makevec{b}}
\def \bzero {\makevec{0}}

\DeclareMathOperator{\IG}{IG}
\DeclareMathOperator{\CR}{cr}
\DeclareMathOperator{\PCR}{pcr}
\newcommand\RCR{\overline{\CR}}
\DeclareMathOperator{\OCR}{ocr}
\DeclareMathOperator{\SEG}{SEG}
\DeclareMathOperator{\STRING}{STRING}
\DeclareMathOperator{\CONV}{CONV}
\DeclareMathOperator{\DIR}{DIR}
\DeclareMathOperator{\vcong}{vcong}
\DeclareMathOperator{\econg}{econg}
\DeclareMathOperator{\vspars}{vspars}
\DeclareMathOperator{\espars}{espars}

\newcommand\ExSym{{\mbox{\bfseries E}}}

\newcommand\Ex[1]{\ExSym\hspace{-0.2ex}\left[#1\right]}

\newcommand{\Prob}[1]{{\rm Prob}\hspace{-0.2ex}\left[ #1 \right]}

\newcommand\dd[1]{\,\mathrm{d}#1}

\definecolor{jgray}{gray}{0.5}
\newcommand\grayhalf{{\color{jgray}\frac12}}

\newcommand\eps{\varepsilon}
\renewcommand\:{\colon}

\newcommand{\heading}[1]{\vspace{1ex}\par\noindent{\bf\boldmath #1}}
\newcommand\defi[1]{{\bf\boldmath #1}}

\graphicspath{{figures/}}

\def\immediateFigure#1{%
\smallskip\begin{center}#1\end{center}\smallskip }

\newcommand{\immfig}[1]  % immediate figure
{\immediateFigure{\mbox{\includegraphics{#1}}}}

\newcommand{\immfigw}[2] % immediate figure with prescribed width
{\immediateFigure{\mbox{\includegraphics[width=#2]{#1}}}}

\newlength{\fparwidth}
\newlength{\myparindent}
\newcommand\lecture[4]{}
\newcommand\jm{}

\newcommand\localbib{\bibliographystyle{alpha}\bibliography{gg}}

\begin{document}
\maketitle
\begin{abstract} String graphs, that is, intersection graphs of curves 
in the plane, have been studied since the 1960s. We provide an expository
presentation of several results, including very recent ones:
some string graphs require an exponential number of crossings
in every string representation; exponential number is always sufficient;
string graphs have small separators; and
the current best bound on the crossing number of a graph in terms 
of pair-crossing number.
For the existence of small separators, the
proof includes generally useful
results on approximate flow-cut dualities.
\end{abstract}

This expository paper was prepared as a material for
two courses co-taught by the author in 2013, at  Charles University
and at ETH Zurich. It aims at a complete and streamlined
presentation of several results concerning 
\emph{string graphs}.
This important and challenging class of intersection graphs has
traditionally been studied at the Department of Applied 
Mathematics of the Charles University, especially by Jan Kratochv\'il
and his students and collaborators.

A major part of the paper is devoted to
a separator theorem by Fox and Pach, recently improved
by the author, as well as an application of it by T\'oth
in a challenging problem from graph drawing, namely,
bounding the crossing number by a function of the pair-crossing number.
This is an excellent example of a mathematical proof with a simple idea
but relying on a number of other results from different areas. 
The proof is presented in full, assuming very little as a foundation, 
so that the reader can see everything that is involved. 
A key step is an approximate flow-cut duality from combinatorial
optimization and approximation algorithms, whose proof
relies on linear programming duality and a theorem on metric embeddings.

\heading{Acknowledgments. } I am very grateful to Rado Fulek, Vincent Kusters,
Jan Kyn\v{c}l, and Zuzana Safernov\'a for proofreading, comments, 
and corrections. It was a pleasure to teach the courses together
with Pavel Valtr in Prague and with Michael Hoffmann and Emo Welzl
in Zurich, and to work with Jan Kratochv\'il on questions
in string graphs as well as on many other things. I also thank
an anonymous referee for numerous useful remarks and suggestions.

% READERS: Kyncl Safer SENT sec. 1-8 2/4/2013
% KYNCL ACK!!!
% ACK BASU ROY SCHAEFER STEFANKOVIC KUSTERS
% Fulek Kolman cela kap. 13/4/2013

% more exercises: 
% exponential coordinates in 3-DIR ???

\lecture{??}{04}{2013}{String graphs, separators, and approximate duality}{\jm}\label{ch:string}

\section{Intersection graphs}\label{s:ig}

\heading{The classes $\IG(\MM)$. }
Let $\MM$ be a system of sets; we will typically consider systems of
geometrically defined subsets of $\R^2$, such as all segments 
in the plane. We define $\IG(\MM)$, the class of \defi{intersection graphs}
of $\MM$, by
\[%\begin{eqnarray*}
\IG(\MM)= %&=&
\Big\{(V,E): V\subseteq\MM, %\\
%&&\ \ \ 
E=\{\{M,M'\}\in {\textstyle {V\choose 2}}: M\cap M'\ne\emptyset\}\Big\}.
%\end{eqnarray*}
\]
In words, the vertices of each graph in $\IG(\MM)$ are
sets in $\MM$, and two vertices are connected by an
edge if they have a nonempty intersection. 

Usually we consider intersection graphs of $\MM$ up to isomorphism;
 i.e., we regard a graph $G$ as an intersection graph of $\MM$ if it is 
merely isomorphic to a graph $G'\in\IG(\MM)$. In that case
we call $V(G')\subseteq\MM$ an \defi{$\MM$-representation}
of $G$, or just a \defi{representation} of $G$ if $\MM$ is understood.

\heading{Important examples. }
\begin{itemize}
\item For $\MM$ consisting of all (closed) intervals on the real line,
we obtain the class of \defi{interval graphs}. This is one of the most
useful graph classes in applications. Interval graphs have several 
characterizations, they can be recognized in linear time, and 
there is even a detective story \emph{Who Killed the Duke of Densmore?}
by Claude Berge (in French; see \cite{densmore} for
English translation) in which the solution depends on properties
of interval graphs.
%\item If we restrict to unit-length intervals, we get \emph{unit interval
%graphs}, a subclass of interval graphs.
\item  \defi{Disk graphs}, i.e.,
intersection graphs of disks in the plane,
and \defi{unit disk graphs}
have been studied extensively. 
Of course, one can also investigate intersection
graphs of \emph{balls in $\R^d$} for a given $d$, or of \emph{unit balls}.
\item Another interesting class is $\CONV$, the intersection graphs
of convex sets in the plane. 
\item Here we will devote most of the time to the class $\STRING$
of \defi{string graphs}, the intersection graphs of simple curves in the
plane.
\item Another important class is
$\SEG$, the \defi{segment graphs}, which are the intersection graphs of
line segments in~$\R^2$.
\end{itemize}
Other interesting classes of graphs are obtained by placing various
restrictions on the mutual position or intersection
pattern of the sets representing the
vertices. For example:
\begin{itemize}
\item For an integer $k\ge 1$,
$k$-$\STRING$ is the subclass of string graphs consisting of
all graphs representable by curves such that every two of them
have at most $k$ points of intersection.\footnote{Some authors 
moreover require that each of the intersection points is
a \emph{crossing}, i.e., a point where, locally, one of the edges passes
from one side of the second edge to the other (as opposed
to a \emph{touching point}).}
\item For $k\ge 1$, the class $k$-$\DIR$ consists of the segment graphs
possessing a representation in which the segments involved have at most
$k$ distinct directions. (So $1$-$\DIR$ are just interval graphs.)
\item The \defi{kissing graphs
of circles}, sometimes also called \defi{contact graphs
of circles} or \defi{coin graphs},
 are disk graphs that admit a representation by disks
with disjoint interiors; that is, every two disks either are disjoint
or just touch. The beautiful and surprisingly useful 
\defi{Koebe--Andreev--Thurston theorem} asserts that 
a graph is a kissing graph of circles if and only if it is planar.
While ``only if'' is easy to see, the ``if'' direction is
highly nontrivial. Here we have just mentioned this gem of a result but
we will not discuss it any further.
\end{itemize}

\heading{Typical questions. } For each class $\CC$
of intersection graphs, and in particular,
for all the classes mentioned above, 
one can ask a number of basic questions. Here are some examples:
\begin{itemize}
\item How hard, computationally, is the \defi{recognition problem for $\CC$}?
That is, given an (abstract) graph $G$, is it isomorphic to a graph 
in $\CC$? For some classes, such as the interval graphs, polynomial-time
or even linear-time algorithms have been found, while for many other
classes the recognition problem has been shown NP-hard, and sometimes
it is suspected to be even harder (not to belong to NP).
\item How complicated a representation may be required for the graphs
in $\CC$?  In more detail, we first need to define some reasonable notion
of \defi{size} of a representation of a graph in $\CC$.
 Then we ask, given an integer $n$,
what is the maximum, over all $n$-vertex graphs  $G\in \CC$, of the
smallest possible size of a representation of $G$? 

For example, it is not too difficult to show that each segment graph
has a representation in which all of the segments have endpoints
with integral coordinates. For such a representation, the size
can be defined as the total number of bits needed for encoding
all the coordinates of the endpoints.
\item Can the chromatic number be bounded in terms of
the clique number? It is well known that there are
graphs $G$ with clique number $\omega(G)=2$, i.e., triangle-free,
and with chromatic number $\chi(G)$ arbitrarily large. 
On the other hand, many important classes, such as interval
graphs, consist of \emph{perfect graphs}, which satisfy
$\omega(G)=\chi(G)$. Some classes $\CC$ display an intermediate
behavior, namely, $\chi(G)\le f(\omega(G))$ for all $G\in\CC$
and for some function $f\:\N\to\N$; establishing a bound of this
kind is often of considerable interest, and then one may ask
for the smallest possible $f$. If $\chi(G)$ cannot be bounded
in terms of $\omega(G)$ alone, one may still investigate
bounds for $\chi(G)$ in terms of $\omega(G)$ and the
number of vertices of~$G$.
\item One may also consider two classes of interest, $\CC$ and $\CC'$,
and ask for inclusion relations among them (e.g., whether $\CC\subseteq\CC'$,
or $\CC'\subseteq\CC$, or even $\CC=\CC'$). Some relations are quite easy,
such as $\SEG\subseteq\CONV\subseteq \STRING$, but others may be
very challenging. 

For example, Scheinermann conjectured in his PhD.~thesis
in 1984 that all planar graphs are in $\SEG$. It took over
20 years until Chalopin, Gon{\c{c}}alves, and Ochem \cite{1string}
managed to prove the weaker result that all planar graphs are
in $1$-$\STRING$, and in 2009 Chalopin and Gon{\c{c}}alves 
\cite{ChalopinGoncalves} finally established Scheinermann's conjecture.
\end{itemize}
Thus, it should already be apparent from our short lists of classes 
and questions that the study of intersection graphs is an area in which 
it is very easy to produce problems (and exercises). However, instead
of trying to survey the area, we will focus on a small number of 
selected results. Some of them also serve us as a stage
on which we are going to show various interesting tools in action.

\begin{exercise}\label{ex:seg-integral}
 Prove carefully an assertion made above:
every $\SEG$-graph has a representation with all segment endpoints
integral. (Hint: check the definition of $\SEG$ again and note
what it does \emph{not} assume.)
\end{exercise}

\begin{exercise} Show that graphs in $100$-$\STRING$ can be recognized in
NP.
\end{exercise}

\section{Basics of string graphs}

We begin with a trivial but important observation: all of the complete
graphs $K_n$ are string graphs. Hence, unlike classes such as planar
graphs, string graphs can be dense and they have no forbidden minors.
Moreover, they are not closed under taking minors;
thus, the wonderful Robertson--Seymour theory is not applicable.

Another simple observation asserts that every planar graph
is a string graph, even $2$-$\STRING$. The following picture
indicates the proof:
\immfig{j-pl-2string}
As we have mentioned, every planar graph is even a segment graph,
but this is a difficult recent result~\cite{ChalopinGoncalves}.

\begin{example}\label{ex:k5sd}
It is not completely easy to come up with an example of a non-string
graph. Here is one:
\immfig{j-k5sd}
(more generally, every graph obtained from a non-planar graph
by replacing each edge by a path of length at least two is non-string).
\end{example}

\begin{proofhd}{Sketch of proof}
For contradiction we suppose that this graph has 
a representation by simple curves (referred
to as \emph{strings} in this context), where each $v_i$
is represented by a string $\gamma_{i}$ and $v_{ij}$
is represented by $\gamma_{ij}$. From such a representation we will
obtain a planar drawing of $K_5$, thus reaching a contradiction. 

To this end, we first select,
for each $i<j$, a piece $\pi_{ij}$ of $\gamma_{ij}$ connecting a point
of $\gamma_i$ to a point of $\gamma_j$ and otherwise disjoint from
$\gamma_i$ and $\gamma_j$. Next, we continuously shrink each $\gamma_i$
to a point, pulling the $\pi_{ij}$ along---the result is the promised
planar drawing of~$K_5$. The picture shows this construction in the vicinity
of the string $\gamma_1$:
\immfig{j-vshrink}
\end{proofhd}

Admittedly, this argument is not very rigorous, and if the strings are
arbitrary curves, it is difficult to specify the construction precisely.
An easier route towards a rigorous proof hinges on the following
generally useful fact.

\begin{lemma}\label{l:PL}
 Every (finite) string graph $G$ can be represented by
polygonal curves, i.e., simple curves consisting of finitely
many segments. We may also assume that every two curves have finitely
many intersection points, 
and that no point belongs to three or more curves.
\end{lemma}

\begin{proofhd}{Sketch of proof} We start from an arbitrary string 
representation of $G$. By compactness, there exists an $\eps>0$
such that every two disjoint strings in the representation
have distance at least $\eps$. For every two strings $\gamma,\delta$ that
intersect, we pick a point $p_{\gamma\delta}$ in the intersection.
Then we replace each string $\gamma$ by a polygonal curve
that interconnects all the points $p_{\gamma\delta}$ and
lies in the open $\frac\eps2$-neighborhood of~$\gamma$.

By a small perturbation of the resulting polygonal curves
we can then achieve finitely many intersections and eliminate
all triple points.
\end{proofhd}

\begin{exercise} Let $U\subseteq\R^2$ be an open, arcwise connected set;
that is, every two points of $U$ can be connected by a simple curve in~$U$.
Prove, as rigorously as possible, that every two points of $U$
can also be connected by a polygonal curve.
\end{exercise}

Let us call a string representation as in Lemma~\ref{l:PL}
\defi{standard}.

\section{String graphs requiring exponentially many intersections}
\label{s:string-exp}

\heading{How hard is to recognize string graphs?}
Using an ingenious reduction, Kratochv\'il \cite{kratoch-string-hard}
proved that recognizing string graphs is NP-hard, but the question
remained, does this problem belong to the class NP? 

A natural way of showing membership of the problem in NP would be 
to guess a string representation, and verify in
polynomial time that it indeed represents a given graph~$G$.
A simple way of specifying a string representation is to 
put a vertex into every intersection point of the strings, and describe the
resulting plane (multi)graph:
\immfig{j-stri-pl}
In this description, the edges are labeled by the strings they come from.
Then it can be checked whether such a plane graph indeed provides
a string representation of~$G$. 

This argument may seem to prove membership
in NP easily, but there is a catch: namely, we would need to know
that there is a polynomial $p(n)$ such that every string graph
on $n$ vertices admits a string representation with at most $p(n)$
intersection points. However, as was noticed in \cite{km-exponential},
this is false---as we will prove below,
there are string graphs for which every
representation has \emph{exponentially many intersections}.
After this result, for ten years it was not clear whether
there is \emph{any algorithm at all} for recognizing string graphs.

\heading{Weak realizations. } As an auxiliary device, we introduce
the following notions. An \defi{abstract topological graph}
is a pair $(G,R)$, where $G$ is an (abstract) graph and $R\subseteq
{E(G)\choose 2}$ is a symmetric relation on the edge set. A \defi{weak
realization} of such $(G,R)$ is a drawing of $G$ in the plane
such that whenever two edges $e,e'$ intersect (sharing a vertex
does not count), we have $\{e,e'\}\in R$.
Thus, $R$ specifies which pairs of edges are allowed (but not forced)
to intersect.

We call a weak realization 
\defi{standard} if the corresponding drawing
of $G$ is \defi{standard}, by which we mean that
the edges are drawn as polygonal curves, every two
intersect at finitely many points, and no three edges have a common 
intersection (where sharing a vertex does not count).
(Moreover, as in every graph drawing we assume that the edges
do not pass through vertices.)
Standard drawings help us to get rid of ``local'' difficulties in proofs.

\begin{exercise}\label{ex:wPL}
{\rm (a)}  Prove that if $(G,R)$ has a weak realization, then
it also has a standard weak realization.
(This is analogous to Lemma~\ref{l:PL}, but extra care is needed
near the vertices!)

(b) Prove that if $(G,R)$ has a weak realization $W$ with finitely
many edge intersections in which no three edges have a common
intersection, then it also has a  standard weak realization $W'$ with 
at most as many edge intersections as in~$W$.
% Convert to planar graph and use Fary
\end{exercise}

For a string graph $G$, let $f_s(G)$ denote the minimum number of 
intersection points in a standard string representation of $G$, and let 
\[
f_s(n):=\max\{f_s(G): G\mbox{ a string graph on $n$ vertices}\}.
\]
Similarly, for an abstract topological graph $(G,R)$ admitting a weak
realization, let $f_w(G,R)$ be the minimum number of edge intersections in
a standard weak realization
%\footnote{In order to get good definitions
%of $f_s$ and $f_w$, we should also insist that three or more
%strings or three or more edges never have a common point. 
%However, for our purposes this condition will not be important.}
 of $(G,R)$, and 
\[
f_w(m):=\max\{f_w(G,R): (G,R)\mbox{ weakly realizable, } |E(G)|=m\}.
\]

\begin{observation}\label{o:fwfs}
 $f_w(m)\le f_s(2m)$.
\end{observation}

\begin{proof} Let $(G,R)$ be an abstract topological graph 
with $m$ edges witnessing
$f_w(m)$. We may assume that $G$ is connected and non-planar (why?),
and thus $m\ge n=|V(G)|$. 

We consider a (standard)
weak realization $W$ of $(G,R)$ with $f_w(m)$ intersections,
and construct a string representation of a string graph $H$ as follows:
we replace every vertex in $W$ by a tiny \emph{vertex string},
and every edge by an \emph{edge string}, as is indicated below:
\immfig{j-fwfs}
This $H$ has $m+n\le 2m$ vertices, and using the monotonicity
of $f_s$, it suffices to show that
$f_s(H)\ge f_w(m)$. This follows since a string representation of $H$ 
with $x$ intersections yields
a weak realization of $(G,R)$ with at most $x$ intersections,
by contracting the vertex strings
to points and pulling the edge strings along (this is the same argument
as in Example~\ref{ex:k5sd}). 
\end{proof}

\begin{exercise}\label{ex:fsfw} Prove that
$f_s(n) \le f_w(n^2)+n^2$. (Or $f_s(n)\le f_w(O(n^2))+O(n^2)$
if this looks easier.)
% for every intersecting pair pick vertex at an intersection
\end{exercise}

\begin{theorem}\label{t:expo}
 There is a constant $c>0$ such that
$f_w(m)\ge 2^{cm}$, and consequently, $f_s(n)\ge 2^{(c/2)n}$.
\end{theorem}

\begin{proof} For $k\ge 1$, we define a planar graph $P_k$ according
to the following picture:
\immfig{j-km1}
($P_{k}$ is obtained from $P_{k-1}$ by adding vertices $u_{k}$
and $v_{k}$ to the left and right of $u_{k-1}$, respectively, and
adding the vertical edges $\{u_{k},u'_k\}$
and $\{v_{k},v'_k\}$). Then we create an abstract topological
graph $(G_k,R_k)$ from $P_k$: $G_k$ is obtained from $P_k$ by adding the
edges $\{u_1,v_1\},\ldots,\{u_k,v_k\}$, and the relation $R_k$ allows
each of the edges $\{u_i,v_i\}$ to intersect
 all of the edges drawn dashed in the
picture above.
%, \emph{except} for the edges $\{u_i,u_{i-1}\}$ and
%$\{u_{i-1},v_i\}$ (for $i\ge2$). 
No other edge intersections are permitted.

Each $(G_k,R_k)$ has a weak realization:
\immfig{j-km2}
We prove by induction on $i$ that in every weak realization
of $G_k$, the edge $\{u_i,v_i\}$ intersects $\{a,b\}$
at least $2^{i-1}$ times, $1\le i\le k$; 
then the theorem will follow.

Since $P_k$ is a  3-connected graph, it has a topologically
unique drawing. From this the case $i=1$ can be considered obvious.
For $i\ge 2$, the situation for the edge $\{u_i,v_i\}$ looks, after
contracting the edge $\{u_{i-1},u_{i-1}'\}$ and
a simplification preserving the topology, as follows:
\immfig{j-km3}
Thus, the edge $\{u_i,v_i\}$ has to cross $\{v_{i-1},v'_{i-1}\}$. 

Now we will use the drawing of $\{u_i,v_i\}$ to get two different curves
$\pi_1,\pi_2$
that both ``duplicate'' the previous edge $\{u_{i-1},v_{i-1}\}$.
The first curve $\pi_1$ starts at $u_{i-1}$ and follows $\{u_{i-1},u_i\}$
up to the point where $\{u_i,v_i\}$ intersects $\{u_{i-1},u_i\}$
the last time before hitting $\{v_{i-1},v'_{i-1}\}$ (that point
can also be $u_i$). 
Then $\pi_1$ follows $\{u_i,v_i\}$ almost up to the first intersection
with $\{v_{i-1},v'_{i-1}\}$, and finally, it goes very near 
$\{v_{i-1},v'_{i-1}\}$ until $v_{i-1}$:
\immfig{j-km4}
If we remove the drawings of the
edges $\{u_j,v_j\}$, $j\ge i-1$, from the considered weak realization
of $G_k$, and add $\pi_1$ 
as a new way of drawing  the edge $\{u_{i-1},v_{i-1}\}$,
 we obtain a weak realization of  $G_{i-1}$. Therefore, by the
inductive hypothesis, $\pi_1$ crosses $\{a,b\}$ at least $2^{i-2}$ times.

Similarly we construct $\pi_2$, disjoint from $\pi_1$,
which starts at $u_{i-1}$ and first follows $\{u_{i-1},v_i\}$.
It again has to cross $\{a,b\}$
at least $2^{i-2}$ times, and the induction step is finished.
\end{proof}

\section{Exponentially many intersections suffice}

The first algorithm for recognizing string graphs was provided
by Schaefer and \v{S}tefankovi\v{c} \cite{SchaeferStefankovic},
who proved an upper bound on the number of intersections
sufficient for a representation of every $n$-vertex
string graph. Similar to the previous section, their proof
works with weak representations.

\begin{theorem}[\cite{SchaeferStefankovic}]\label{t:exp-ub}
We have $f_w(m)\le m2^m$. Consequently (by Exercise~\ref{ex:fsfw}),
$f_s(n)=2^{O(n^2)}$.
\end{theorem}

This result 
implies, by the argument given at the beginning of 
Section~\ref{s:string-exp}, that string graphs can be recognized in NEXP
(nondeterministic exponential time). 

Later Schaefer, Sedgwick, and \v{S}tefankovi\v{c} \cite{string-np} 
proved that string graphs can even be  recognized in NP. The main idea of their
ingenious argument is that, even though a string representation may
require exponentially many intersections, there is always a
representation admitting a compact encoding,
of only polynomial size, by something like a context-free grammar.
They also need to show that, given such a compact encoding of a 
collection of strings, one can verify in polynomial time whether it 
represents a given graph. We will not discuss their proof any further
and we proceed with a proof of Theorem~\ref{t:exp-ub}.

Let $(G,R)$ be a weakly realizable abstract topological graph with
$m$ edges. It has a standard weak realization 
(edges are polygonal curves with finitely
many intersections, and no triple intersections;
see Exercise~\ref{ex:wPL}).  Moreover, we can make sure that 
the edges  cross at every intersection point,
since a ``touching point'' \includegraphics{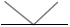} can
be perturbed away: \includegraphics{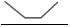} (note the advantage
of working with weak realizations, in which we need not worry
about losing intersections). 

Theorem~\ref{t:exp-ub} is an immediate consequence of the
following claim: \emph{if $W$ is a standard weak realization
in which some edge $e$ has at least $2^m$ crossings,
 then there is another standard weak realization $W'$ with fewer crossing 
than in~$W$.}

\begin{lemma}\label{l:subedge} If an edge $e$ has at least $2^m$ crossings,
then there is a contiguous segment $\hat e$ of $e$ that contains at
least one crossing and that crosses every edge of $G$ an \emph{even}
number of times.
\end{lemma}

\begin{proof} The lemma is an immediate consequence of the
following combinatorial statement: \emph{If $w$ is a word
(finite sequence) of length $2^m$ over an $m$-letter alphabet
$\Sigma$, then there is a nonempty subword (contiguous subsequence)
$x$ of $w$ in which each symbol of $\Sigma$ occurs an even number
of times.}

To prove this statement, 
let us define, for  $i=0,1,\ldots,2^m$,
 a mapping $f_i\:\Sigma\to \{0,1\}$, where $f_i(a)=0$
if $a$ occurs an even number of times among the first $i$ symbols of $w$,
and $f_i(a)=1$ otherwise.
Since there are only $2^m$ distinct mappings $\Sigma\to \{0,1\}$,
there are two indices $i\ne j$ with $f_i=f_j$. Then the subword
of $w$ beginning at position $i+1$ and ending at position $j$
is the desired~$x$.
\end{proof}

Now we fix $e$ and $\hat e$ as in the lemma. 
We can deform the plane
by a suitable homeomorphism so that $\hat e$ is a horizontal straight
segment and there is a narrow band along it,
which we call the \emph{window}, in which the edges crossing $\hat e$
appear as little vertical segments, and in which no other portions
of the edges are present:
\immfig{j-win1}
Any edge $f\ne e$ has an even number $2n_f$ crossings with $\hat e$,
and it intersects the border of the window in $4n_f$ points.
Let us number these $4n_f$ points as $p_{f,1},\ldots,p_{f,4n_f}$
in the order as they appear along $f$
(we choose one of the two possible directions of traversing $f$
arbitrarily). 

Here is the procedure for redrawing the original weak representation
$W$ into $W'$ so that the total number of crossings is reduced.
\begin{enumerate}
\item We apply a suitable homeomorphism of the plane that maps
the window to a circular disk with $\hat e$ as the horizontal diameter,
while the edges crossing $\hat e$ still appear as vertical segments
within the window. For every $f$ and every $i=1,2,\ldots,2n_f-1$,
the point $p_{f,2i}$ is connected to $p_{f,2i+1}$ by an arc of $f$
outside the window. Let us call these arcs for $i$ odd 
the \emph{odd connectors} and for $i$ even
the \emph{even connectors}.  The following illustration
shows only one edge $f$, although the window may be intersected by
many edges. The points $p_{f,i}$ are labeled only by their indices,
and the odd connectors are drawn thick:
\immfig{j-win2}
\item We erase everything inside the window. Then we map the odd connectors
inside the window by the circular inversion that maps the outside
of the window to its inside,
%\footnote{Circular inversion is used just
%for convenience; we could do with any homeomorphism mapping
%the outside of the window to the inside minus one interior point,
%and fixing the boundary.}
 while the even connectors
stay outside. Crucially, two odd connectors that did not intersect
before the circular inversion still do not intersect.
Next, we apply the mirror reflection
about $\hat e$ inside the window to the odd connectors:
\immfig{j-win3}
As the picture illustrates, these transformed odd connectors together
with the original even connectors connect up the initial piece of $f$
to the final piece. This new way of drawing of $f$ crosses the window
$n_f$ times, only half of the original number. This is the moment
where we use the fact that each edge crosses $\hat e$ an even number of times;
otherwise, the re-connection of $f$ would not work.

After this redrawing of all the edges crossing $\hat e$,
edges that did not cross before still do not cross (while some
intersections may be lost).
\item It remains to draw the erased portion of $e$. We do not want to
draw it horizontally, since we would have no control over the intersections
with the transformed odd connectors. Instead, we draw it along the top
or bottom half-circle bounding the window, whichever gives a smaller
number of intersections (breaking a tie arbitrarily). 
\immfig{j-win4}
Each $f$ crosses the window border $2n_f$ times after the redrawing,
and thus one of the half-circles is crossed at most
$\sum_f n_f$ times---while originally $\hat e$ was crossed $2\sum_f n_f$
times. Hence the redrawing indeed reduces the number
of crossings. The resulting weak realization
is not necessarily standard, but we can make
it standard without increasing the number of intersections
(Exercise~\ref{ex:wPL}(b)), and Theorem~\ref{t:exp-ub} is proved.
\end{enumerate}

\section{A separator  theorem for string graphs}

Let $G$ be a graph. A subset $S\subseteq V(G)$ is called
a \defi{separator} if there is a partition of $V(G)\setminus S$
into disjoint subsets $A$ and $B$ such that there are no edges
between $A$ and $B$ and $|A|,|B|\le \frac23|V(G)|$.

\begin{exercise} Check that the above definition of a separator
is equivalent to requiring all connected components of $G\setminus S$ 
to have at most $\frac 23|V(G)|$ vertices.
\end{exercise}

\begin{exercise}{\rm (a)} Show that every tree has a one-vertex separator.

{\rm (b)} We could also define a $\beta$-separator, for 
$\beta\in (0,1)$, by replacing $\frac23$ in the above definition
by $\beta$. Check that for $\beta<\frac23$, there are trees with
no one-vertex $\beta$-separator. From this point of view,
the value $\frac23$ is natural. (For most applications, though,
having  $\beta$-separators for a constant $\beta<1$ is sufficient,
and the specific value of $\beta$ is not too important.)
\end{exercise} 

\emph{Separator theorems} are results asserting that all graphs
in a certain class have ``small'' separators (much smaller than
the number of vertices). They have lots of applications, 
and in particular, they are the basis of many efficient divide-and-conquer
algorithms.

Probably the most famous separator theorem, and arguably one of the
nicest and most useful ones, is the \defi{Lipton--Tarjan separator theorem
for planar graphs}, asserting that \emph{every planar graph on $n$
vertices has a separator of size $O(\sqrt n\,)$.}

\begin{exercise}
 Show that the $m\times m$ square grid has no
separator of size $m/4$. Thus, the order of magnitude in
the Lipton--Tarjan theorem cannot be improved.
% free rows  + free column
\end{exercise}

The separator theorem has several proofs (let us mention a simple
graph-theoretic proof by Alon et al. \cite{ast-ps-94}
and a neat proof from the Koebe--Andreev--Thurston theorem mentioned
in Section~\ref{s:ig}; see, e.g., \cite{pa-cg-95}). There are
scores of generalizations and variations. For example,
every class of graphs with a fixed excluded minor admits
$O(\sqrt n\,)$-size separators \cite{ast-ps-94}.

Here we focus on a separator theorem for string graphs.
Of course, in this case we cannot bound the separator size 
by a sublinear function of the number of vertices
(because $K_n$ is a string graph); for this reason,
the bound is in terms of the number of \emph{edges}.

\begin{theorem}\label{t:strsep}
Every string graph with $m\ge 2$ edges has a separator with
$O(\sqrt m\,\log m)$ vertices.
\end{theorem}

The first separator theorem for string graphs, with
a worse bound of $O(m^{3/4}\sqrt{\log m}\,)$, was proved
by Fox and Pach \cite{fox2010separator}. The improvement
to $O(\sqrt m\,\log m)$ was obtained while preparing this text,
and it was published in a concise form in~\cite{mat-stringsep}.

The proof, whose exposition will occupy most of the
rest of this chapter, is a remarkable chain of diverse ideas
coming from various sources. 
%The structure of the proof will
%be outlined at the end of Section~\ref{s:approxdual}.

Fox and Pach conjectured that the theorem should hold with
$O(\sqrt m\,)$. This, if true, would be asymptotically optimal:
we already know this for graphs with $n$ vertices and $O(n)$
edges, since every planar graph is
a string graph, but asymptotic optimality
 holds for graphs with any number of edges
between $n$ and~$n\choose 2$.

The separator theorem shows that string graphs are (globally)
very different from typical (not too dense) 
random graphs and, more generally, from expanders. 

In the next section, we demonstrate a surprising use of
Theorem~\ref{t:strsep}; for a number of other applications
we refer to  \cite{fox2010separator,foxpach-stringappl}.
Then we start working on the proof of Theorem~\ref{t:strsep}
in Section~\ref{s:approxdual}.

\section{Crossing number versus pair-crossing number}

Now something else: we will discuss crossing numbers in this section.
The \defi{crossing number} $\CR(G)$ of a graph $G$ is the smallest possible
number of edge intersections  (crossings) in a standard drawing of $G$.
We recall that in a standard drawing, 
edges are polygonal lines with finitely many intersections and no triple
points; see Section~\ref{s:string-exp}. In this section we consider
only standard drawings.

One may also consider
the \defi{rectilinear crossing number} $\RCR(G)$, which is the minimum
number of crossings in a straight-edge drawing of $G$, but this behaves
very differently from $\CR(G)$, and the methods involved in its study are
also different from those employed for the crossing number.

\heading{An algorithmic remark. }
The crossing number is an extensively studied and difficult graph parameter.
Let us just mention in passing that computing $\CR(G)$ is 
known to be NP-complete, but a
tantalizing open problem is, how well it can be approximated 
(in polynomial time). On the one hand, there is a constant $c>1$
such that $\CR(G)$ is hard to approximate within a factor of $c$ 
\cite{cabello-crhard}, and on the other hand, there is a (highly complicated)
algorithm \cite{chuzhoy2011algorithm} with approximation factor
roughly $n^{9/10}$ for $n$-vertex graphs with maximum degree bounded
by a constant. The latter result may not look very impressive, but it is the
first one breaking a long-standing barrier of $n$. 

A difficult case here are graphs with relatively small, but not too
small, crossing number, 
say around~$n$. Indeed, on the one hand, for every fixed $k$, there
is a linear-time algorithm deciding whether $\CR(G)\le k$
\cite{kawarabayashi2007computing,grohe-cr}. On the other hand,
for a graph with maximum degree bounded by a constant and with crossing
number $k$, a drawing with at most $O((n+k)(\log n)^2)$ crossings
can be found in polynomial time; this is based on \cite{LeightonRao},
with improvements of \cite{evenguhas,ARV} (also see \cite{ChuMaSi}).

\heading{The single-crossing lemma. }
Here is a useful basic fact about drawings minimizing the crossing number.

\begin{lemma}[Single-crossing lemma]
\label{l:onecross} In every (standard) drawing of $G$ 
that minimizes the crossing number, no two edges intersect more than
once.
\end{lemma}

\begin{proof} We show that if edges $e,e'$ intersect at least twice,
the number of crossings can be reduced. We consider crossings $X_1$
and $X_2$ that are consecutive along $e$. There are two cases
to consider, the second one being easy to overlook, and in each
of them we redraw the edges locally as indicated:
\immfig{j-twocross}
This reduces the total number of crossings by at least $2$. We also note
that since the part of $e'$ between $X_1$ and $X_2$ may be intersected
by $e$, this redrawing may introduce self-intersections of $e$---but these
are easily eliminated by shortcutting the resulting loops on~$e$.
\end{proof}

\begin{exercise} Let us say that $X_1,X_2$ is a \emph{simple pair
of crossings} of edges $e$ and $e'$ if $X_1$ and $X_2$ are consecutive
on both $e$ and $e'$. Draw an example of two edges that cross
several times but that have no simple pair of crossings.
% inclusion-minimal region delimited by the segments between X1 and X2
\end{exercise}

The lemma just proved shows that it does not matter whether we
define the crossing number as the minimum number of \emph{crossings}
or as the minimum number of \emph{crossing pairs of edges} in a drawing.
Do you believe the previous sentence? If yes, you are in a good company,
since many people got caught, and only Pach and T\'oth \cite{PachToth-cr},
and independently Mohar (at a 1995 AMS Conference on Topological Graph Theory),
noticed that the lemma proves nothing like that, since it is not clear
whether a drawing with minimum number of crossings also has
a minimum number of crossing pairs. Indeed, the number of crossing
pairs may very well increase in a redrawing as in the proof of
the lemma:
\immfig{j-newpair}

Thus, it makes good sense to define the \defi{pair-crossing number}
$\PCR(G)$ as the minimum possible number of pairs of edges that cross
in a  drawing of~$G$. Clearly, $\PCR(G)\le\CR(G)$ for all $G$.

It is generally conjectured that $\PCR(G)=\CR(G)$ for all $G$, but
if true, this is unlikely to be proved by a ``local'' redrawing
argument in the spirit of Lemma~\ref{l:onecross}---at least, many
people tried that and failed. 

A warning example is a result of
Pelsmajer et al.~\cite{pelsmajer-al-odd}, who found a graph $G$
with $\OCR(G)<\CR(G)$. Here $\OCR(G)$ is the \defi{odd crossing number}
of $G$, which is the minimum number, over  drawings of $G$,
 of pairs of edges that cross \emph{an odd number of times}. A good motivation
for studying the odd crossing number is the famous \defi{Hanani--Tutte theorem},
asserting that if a graph has a drawing in which every two
\emph{non-adjacent} edges cross an even number of times, then it is planar
(see \cite{Schaefer-HT}  for a modern treatment).
In particular, this implies that $\OCR(G)=0$, $\PCR(G)=0$, and $\CR(G)=0$
are all equivalent.

One direction of investigating the pcr/cr puzzle is to bound the crossing
number by some function of the pair-crossing number, and to try to get
as small a bound as possible. We begin with a simple result in this
direction.

\begin{proposition}\label{p:simple-crpcr}
If $\PCR(G)=k$, then $\CR(G)\le 2k^2$.
\end{proposition}

\begin{proof} We fix a drawing $D$ of $G=(V,E)$
witnessing $\PCR(G)$. Let $F\subseteq E$ be the set of edges that
participate in at least one crossing, and let $E_0=E\setminus F$;
thus, the edges of $E_0$ define a plane subgraph in $D$.
We keep the drawing of these edges and we redraw the edges
of $F$ so that every two of them have at most one crossing,
as in  the proof of Lemma~\ref{l:onecross} (we note that $E_0$
does not interfere with this redrawing in any way).
Since $|F|\le 2k$, the resulting drawing has at most ${2k\choose 2}\le 2k^2$
crossings.
\end{proof}

The current strongest bound is based on the separator theorem 
for string graphs.

\begin{theorem}\label{t:toth}
If $\PCR(G)=k\ge 2$, then $\CR(G)=O(k^{3/2}(\log k)^2)$.
\end{theorem}

The following proof is due to T\'oth \cite{toth430better};
he states a worse bound, but the bound above follows immediately
from his proof by plugging in a better separator theorem.

We begin the proof with a variant of the single-crossing lemma
(Lemma~\ref{l:onecross}).

\begin{lemma}[Red-blue single-crossing lemma]\label{l:rbc}
Let $G$ be a graph in which each edge is either red or blue, and
let $D$ be a drawing of $G$. Then there is a drawing $D'$ of $G$
such that the following hold:
\begin{enumerate}
\itemsep=0pt\topsep=0pt
\item[\rm(i)] Each edge in $D'$ is drawn in an arbitrarily small neighborhood
of the edges of~$D$.
\item[\rm(ii)] Every two edges intersect at most once in~$D'$.
\item[\rm(iii)] The number of blue-blue crossings in $D'$ is 
no larger than in~$D$.
\end{enumerate}
\end{lemma}

\begin{proof} While edges $e,e'$ crossing at least twice exist,
we repeat redrawing operations similar to those in Lemma~\ref{l:onecross}.
However, while in that lemma we swapped portions of $e$ and $e'$,
here we keep $e$ fixed and route $e'$ along it,
\immfig{j-twocross1}
or the other way round. We may again introduce 
self-intersections of $e$ or $e'$, but we remove
them by shortcutting loops.

To decide which way of redrawing to use, we 
let $\hat e$ be the portion of $e$ between $X_1$ and $X_2$
(excluding $X_1$ and $X_2$), and similarly for $\hat e'$. 
Let $b$ and $b'$ be the number of crossings of blue edges
with $\hat e$ and $\hat e'$, respectively, and similarly for $r,r'$.
If the pair $(b,r)$ is lexicographically smaller than $(b',r')$,
%which we write $(b,r)<_{\rm lex} (b',r')$,
we route $\hat e'$
along $\hat e$, and otherwise, $\hat e$ along~$\hat e'$.

It is clear that \emph{if} this redrawing procedure terminates, then it
yields a drawing $D'$ satisfying (i) and (ii). To show that it
terminates, and that (iii) holds, it suffices to
check that each redrawing strictly
lexicographically decreases the vector $(x_{BB},x_{RB},x_{RR})$
for the current drawing, where $x_{BB}$ is the total number of
blue-blue crossings, and similarly for $x_{RB}$ and~$x_{RR}$.

To check this, we distinguish three cases. If both $e,e'$ are
blue, then the redrawing decreases $x_{BB}$. If both $e,e'$
are red, then $x_{BB}$ stays the same and either $x_{RB}$
decreases, or it stays the same and $x_{RR}$ decreases.
Finally, if $e$ is blue and $e'$ is red, then either $x_{BB}$
decreases, or it stays the same and $x_{RB}$ decreases.
\end{proof}

In order to prove Theorem~\ref{t:toth}, we will proceed by induction.
The inductive hypothesis is the following strengthening of the theorem.

\begin{claim}\label{cl:redr} 
Let $G=(V,E)$ be a graph, and let $D$
be a drawing of $G$ with $k$ crossing pairs of edges and
with $\ell\ge 2$ edges that have at least one crossing
(thus, $\ell\le 2k$ and $k\le{\ell\choose 2}$).
Then there is a drawing $D'$ of $G$ such that every edge is drawn
in an arbitrarily small neighborhood of the edges in $D$,
and $D'$ has at most $Ak^{3/2}(\log\ell)^2$ crossings,
where $A$ is a suitable constant.
\end{claim}

\begin{proof}
We proceed by induction on $\ell$ (the proof will also establish
the base case $\ell=2$ directly).

As in the proof of Proposition~\ref{p:simple-crpcr}, we first partition
the edge set $E$ into $F$, the edges with crossings, and $E_0=E\setminus F$.
Thus, $|F|=\ell$.

Let us consider the edges of $F$ in the drawing $D$
as strings (where we cut off tiny pieces near the vertices, so that
the strings meet only if the corresponding edges cross). This defines
a string graph with $\ell$ vertices and $k$ edges. 

By Theorem~\ref{t:strsep}, this string graph has a separator of
size at most $C\sqrt k\log k$, with a suitable constant~$C$.
This defines a partition of $F$ into disjoint subsets $F_0,F_1,F_1$;
 we have $\ell_0:=|F_0|\le C\sqrt k\log k$,
$|F_1|,|F_2|\le \frac23 \ell$, and no edge of $F_1$
crosses any edge of $F_2$.  
Let $k_i$ be the number of crossing pairs of edges of $F_i$ in $D$, $i=1,2$;
we have $k_1+k_2\le k$. Let $\ell_i$ be the number of edges of $F_i$
that cross some edge of $F_i$.

Actually, Theorem~\ref{t:strsep} can be applied only if $k\ge 2$,
while in our case, for $\ell=2$ it may happen that $k=1$. But if
$k=1$, we simply put $F_0=F$ and $F_1=F_2=\emptyset$, and proceed
with the subsequent argument.

Next, we apply the inductive hypothesis to the graphs 
$G_1:=(V,F_1)$ and $G_2:=(V,F_2)$ (drawn as in $D$).  This yields
drawings $D'_1$ and $D'_2$ as in the claim. (If $F_1$ has no
crossings then, strictly speaking, the inductive hypothesis cannot be applied,
but then $D'_1$ can be taken as the plane drawing
of $G_1$ inherited from $D$, and similarly for $F_2$.)

For $\ell_i\ge 2$, we can bound the number of crossings in $D'_i$ by 
$Ak_i^{3/2}(\log\ell_i)^2$ according to the inductive hypothesis;
for $\ell_i=0$ there are no crossings (and $\ell_i=1$ is impossible).
We have $\ell_i\le|F_i|\le\frac 23\ell$, and
hence $\log\ell_i\le\log \ell-c_0$, where $c_0=\log\frac 32>0$.
For the subsequent computation, it will be convenient to bound
$(\log\ell_i)^2$ from above in a slightly strange-looking way:
 by $(\log \ell)(\log\ell-c_0)$. The resulting bound
$Ak_i^{3/2}(\log \ell)(\log\ell-c_0)$ gives $0$ for $k_i=0$ and
so it is also valid in the case $k_i=\ell_i=0$.

We overlay $D'_1$ and $D'_2$ and add the edges
of $F_0$ drawn as in $D$; this gives
a drawing $\tilde D$ of the graph $(V,F)$.
Let us color the edges of $F_1\cup F_2$ blue and those of $F_0$
red. By the above, and using $k_1^{3/2}+k_2^{3/2}\le k^{3/2}$,
 the total number of blue-blue crossings in $\tilde D$ can be bounded
by
\[%\begin{eqnarray*}
A(k_1^{3/2}+k_2^{3/2})(\log\ell)(\log\ell-c_0)
\le Ak^{3/2}(\log\ell)^2- Ac_0k^{3/2}\log\ell.
\]%\end{eqnarray*}
The first term of the last expression is the desired bound for the number of
all crossings, and thus the second term
is the ``breathing room''---we need to  bound
the number of red-blue and red-red crossings by $Ac_0k^{3/2}\log\ell$.

We cannot control the number of red-blue and red-red crossings
in $\tilde D$, but we apply Lemma~\ref{l:rbc} to the graph $(V,F)$
with the drawing $\tilde D$. This provides a new drawing of $(V,F)$,
to which we add the edges of $E_0$ from the original drawing $D$,
and this  yields the final drawing $D'$. The number of blue-blue crossings
in $D'$ is no larger than in $\tilde D$, and the number
of red-blue and red-red crossings is at most
$|F_0|\cdot |F|=\ell_0\ell$.
Using $\ell\le 2k$ and $k\le {\ell\choose 2}\le \ell^2$,
we further bound this by $(C\sqrt k\log k) 2k\le
4Ck^{3/2}\log \ell\le Ac_0k^{3/2}\log\ell$, provided that
$A$ was chosen sufficiently large.

This concludes the inductive proof of the claim and 
thus yields Theorem~\ref{t:toth}.
\end{proof}
% \cleardoublepage

\section{Multicommodity flows, congestion, and cuts}\label{s:approxdual}

We start working towards a proof of the separator theorem
for string graphs (Theorem~\ref{t:strsep}). The overall scheme 
of the proof is given at the end of this section,
but most of the actual work remains for
later sections.

\heading{$s$-$t$ flows.} As a motivation for the subsequent
developments, we briefly recall the duality between flows and cuts in graphs.
If $G=(V,E)$ is a graph with two distinguished vertices $s$ and $t$,
in which every edge has unit capacity,
then the maximum amount of flow from $s$ to $t$ equals 
the minimum number of edges we have to remove in order to
destroy all $s$-$t$ paths. There is also a more general
weighted version, in which the capacity of each edge $e$
is a given real number $w_e\ge 0$.

%This was the simpler and more basic ``edge'' version of the
%flow-cut duality. There is also a ``vertex'' version:
%if the edge capacity is unlimited but every vertex, except for $s$ and $t$, 
%has unit capacity, i.e., allows for at most one unit of flow 
%to pass through, then the maximum $s$-$t$ flow equals
%the minimum number of vertices whose removal destroys all 
%$s$-$t$ paths.

\heading{Multicommodity flows. }
Instead of flows between just two vertices, we will use 
\emph{multicommodity flows}; namely, we want  
a unit flow between every pair $\{u,v\}$ of vertices of the
considered graph.

Let us remark that instead of requiring unit flow for every pair,
we can consider an arbitrary \emph{demand function} $D\:
{V\choose 2}\to [0,\infty)$, specifying some demand $D(u,v)$
on the flow between $u$ and $v$ for every pair $\{u,v\}$.
All of the considerations below can be done in this more general
setting; we can also put weights on edges and vertices.
For simplicity, we stick to the unweighted case, which is
sufficient for us; conceptually, the weighted case mostly
brings nothing new.

For our purposes, it is convenient to formalize a multicommodity
flow as an
assignment of nonnegative numbers to paths in the considered
graph. Thus, we define $\PP$ to be the set of all paths
(of nonzero length) in $G$, and a \defi{multicommodity
flow} is a mapping $\varphi\:\PP\to[0,\infty)$. 
Since we will
talk almost exclusively about multicommodity flows, we will
sometimes say just ``flow'' instead of ``multicommodity flow''.

The amount of flow between two vertices $u,v$ is
$\sum_{P\in\PP_{uv}}\varphi(P)$, where $\PP_{uv}\subseteq \PP$
is the set of all paths with end-vertices $u$ and~$v$.
If we think of the vertices as cities and the edges as roads,
then $\varphi(P)$, $P\in\PP_{uv}$  may be the number of cars per hour driving
from $u$ to $v$ or from $v$ to $u$ along the route~$P$.

\heading{Edge congestion. }
The requirement of unit flow between every pair of vertices
is expressed as $\sum_{P\in\PP_{uv}}\varphi(P)\ge 1$, 
$\{u,v\}\in {V\choose 2}$. We define the \defi{edge congestion}
under $\varphi$ as 
\[
\econg(\varphi)=\max_{e\in E} \sum_{P\in\PP:e\in P}\varphi(P),
\]
and the edge congestion of $G$ as $\econg(G)=\min_\varphi
\econg(\varphi)$, where the minimum is over all 
flows with unit flow between every two vertices.\footnote{A compactness
argument, which we omit, shows that the minimum is actually
attained.} If $G$ is disconnected, then there are no flows $\varphi$
as above, and we have $\econg(G)=\infty$.

Let us consider an \defi{edge cut} in $G$, which for us is a partition
$(A,V\setminus A)$ of $V$ into two \emph{nonempty} subsets.
By $E(A,V\setminus A)$ we denote the set of all edges of $G$ connecting
$A$ to $V\setminus A$. If there is a unit flow between every two vertices
of $G$, then $|A|\cdot|V\setminus A|$ units of flow have to pass through
the edges of $E(A,V\setminus A)$, and hence
\[
\econg(G)\ge \frac{|A|\cdot|V\setminus A|}{|E(A,V\setminus A)|}.
\]
If we define the \emph{sparsity} of the edge cut $(A,V\setminus A)$
as 
\[
\espars(A,V\setminus A)=\frac{|E(A,V\setminus A)|}{|A|\cdot|V\setminus A|},
\]
and the \defi{edge sparsity}\footnote{What we call edge sparsity
is often called just \emph{sparsity}. This quantity is also
closely related to the \emph{Cheeger constant}, or
\emph{edge expansion}, of $G$, which is defined
as $\min_{A\subseteq V: 1\le |A|\le |V|/2} (|E(A,V\setminus A)|/|A|)$.}
$\espars(G):=\min_A \espars(A,V\setminus A)$,
we can write the conclusion of the previous
consideration compactly as $\espars(G)\ge 1/\econg(G)$.

\heading{Approximate duality. } Unlike
in the case of $s$-$t$ flows, it turns out that
the last inequality can be strict.

%\begin{exercise}[The Seymour--Okamura example]
%In the following graph, 
%\end{exercise}

\begin{exercise} 
Let $G$ be an $n$-vertex \defi{constant-degree expander},
which means that, for some constants $\Delta$ and $\beta>0$,
all degrees in $G$ are at most $\Delta$ and $\espars(G)\ge \frac\beta n$.
The existence of such graphs, with some $\Delta,\beta$ fixed
and $n$ arbitrarily large, is well known;
see, e.g., \cite{expander-survey}.
Prove that $\econg(G)>1/\espars(G)$ (assuming that $n$ is sufficiently
large in terms of $\Delta$ and $\beta$), and actually, that
$\econg(G)=\Omega(\frac{\log n}{\espars(G)})$. Hint: show that, say,
half of the vertex pairs have distance $\Omega(\log n)$.
% most pairs have distance about log n
\end{exercise}

However, an important result, discovered by Leighton and Rao
\cite{LeightonRao}, asserts that the gap between the two quantities
cannot be very large; this is an instance of \emph{approximate duality}
between multicommodity flows and cuts.

\begin{theorem}[Approximate duality, edge version]
\label{t:leightonrao}
For every  $n$-vertex graph $G$ we have
 \[\espars(G)=O\Bigl(\frac {\log n}{\econg(G)}\Bigr).\]
\end{theorem}

Although we won't really need this particular theorem,
the proof can serve as an introduction to things
we will actually use, and
we present it in Section~\ref{s:edual}.

\begin{exercise}[Edge sparsity and balanced edge cut]\label{ex:bw}
Let $\beta>0$ and 
let $G$ be a graph on $n$ vertices such that $\espars(H)\le\beta$ for every
induced subgraph $H$ of $G$ on at least $\frac23n$ vertices. Show that
$G$ has a balanced edge cut $(A,V\setminus A)$ with $|E(A,V\setminus A)|\le
\beta n^2$, where balanced means that~$\frac13 n\le |A|\le \frac23 n$.
\end{exercise}

\heading{Vertex notions. } For the proof of the separator theorem for string
graphs, we will need vertex analogs of the ``edge'' notions and results
just discussed.

For a flow $\varphi$ in $G$ we introduce the
\defi{vertex congestion} 
\[
\vcong(\varphi):= \max_{v\in V} \vcong(v),\mbox{ where }
\vcong(v):=\sum_{P\in\PP: v\in P}\grayhalf \varphi(P),
\]
and the vertex congestion of $G$ is $\vcong(G):=\min_\varphi \vcong(\varphi)$,
where the minimum is over all $\varphi$ with unit flow between
every two vertices. The $\grayhalf$ in the above formula should
be interpreted as 1 if $v$ is an inner vertex of the
path $P$, and as $\frac 12$ if $v$ is one of the end-vertices of~$P$.
We thus think of the flow along $P$ as incurring congestion
$\frac12\varphi(P)$ when entering a vertex and congestion
$\frac12\varphi(P)$ when leaving it. (This convention
is a bit of a nuisance in the definition of vertex congestion, but later on, it
will pay off when we pass to a ``dual'' notion.)

By a \defi{vertex cut} in $G$ we mean a partition $(A,B,S)$ of $V$
into three disjoint subsets such that $A\ne\emptyset\ne B$
and there are no edges between 
$A$ and $B$ (this is like in the definition of a separator, except that
we do not require the sizes of $A$ and $B$ to be roughly the same).

If $\varphi$ sends unit flow between every pair of vertices,
then the flows between $A$ and $B$ contribute a total
flow of $|A|\cdot|B|$ through $S$, and moreover, from
each vertex of $S$ we have a flow of $n-1$ to the remaining
vertices. Thus the total congestion of the vertices in $S$
is at least $|A|\cdot|B|+\frac 12 |S|(n-1)$. Losing a constant
factor (and using $n\ge 2$), we bound this somewhat unwieldy
expression from below by $\frac14 |A|\cdot|B|+\frac 14|S|n
=\frac14 |A\cup S|\cdot|B\cup S|$. 

This suggests to define the \emph{sparsity} of a 
vertex cut $(A,B,S)$ as
\[
\vspars(A,B,S):= 
\frac{|S|}{|A\cup S|\cdot|B\cup S|},
\]
and the \defi{vertex sparsity} of $G$
is $\vspars(G):=\min_{(A,B,S)} \vspars(A,B,S)$, where the minimum
is over all vertex cuts.

By the considerations above, we have $\vcong(G)\ge 1/(4\vspars(G))$.
We will need the following analog of Theorem~\ref{t:leightonrao}:

\begin{theorem}[Approximate duality, vertex version]\label{t:vleightonrao}
For every connected $n$-vertex graph $G$ we have
 \[\vspars(G)=O\Bigl(\frac {\log n}{\vcong(G)}\Bigr).\] 
\end{theorem}

The proof is deferred to Section~\ref{s:vdual}.

\begin{exercise}[Vertex sparsity and separators]\label{ex:vsep}
Let $\alpha>0$ and
let $G$ be a graph on $n$ vertices such that $\vspars(H)\le\alpha$ for every
induced subgraph $H$ of $G$ on at least $\frac23n$ vertices. Show that
$G$ has a separator of size at most $\alpha n^2$. 
\end{exercise}

\heading{String graphs have large vertex congestion. }
The last ingredient  in the proof of the 
separator theorem for string graphs is
%by combining Theorem~\ref{t:vleightonrao}, or more precisely,
%the consequence of it in Exercise~\ref{ex:vsep}, with
the following result about string graphs.

\begin{proposition}\label{p:str-vcong}
For every connected string graph $G$ with $n$ vertices and $m$ edges,
we have 
\[ \frac1{\vcong(G)}=O\Bigl(\frac{\sqrt m} {n^2}\Bigr).
\]
\end{proposition}

This is the only specific property of string graphs used in the proof of the
separator theorem. The next section is devoted to the proof
of this proposition.

\begin{proofhd}{Proof of the separator theorem for string graphs
(Theorem~\ref{t:strsep})}
Let $G$ be a string graph with $n$ vertices and $m\ge n$ edges.
We have $1/\vcong(G)=O(\sqrt m/n^2)$ by Proposition~\ref{p:str-vcong}.
By approximate duality (Theorem~\ref{t:vleightonrao}), we have
$\vspars(G)=O((\log n)\sqrt m/n^2)$, and 
so $G$ has a separator of size $O(\sqrt m\,\log n)$
according to Exercise~\ref{ex:vsep}.
\end{proofhd}

\begin{exercise} Let $G$ be a string graph with $m$ edges whose
maximum degree is bounded by a constant $\Delta$. Derive from
Theorem~\ref{t:leightonrao} (the edge version of the approximate duality)
and from Proposition~\ref{p:str-vcong} that $G$ has a separator
of size $O(\sqrt m\,\log m)$, where the implicit constant 
may depend on~$\Delta$.
\end{exercise}

\section{String graphs have large vertex congestion}

The strategy of the proof of Proposition~\ref{p:str-vcong} is this:
Given a string representation of an $n$-vertex graph $G$ and a multicommodity
flow in $G$ with a small vertex congestion, we will construct a drawing of $K_n$
in which only a small number of edge pairs cross.  This will contradict
the following result:

\begin{lemma}\label{l:pcr-Kn} For $n\ge 5$, $\PCR(K_n)=\Omega(n^4)$.
\end{lemma}

The proof of this lemma relies on the following fact.

\begin{fact}\label{f:K5i}
 In every plane drawing of $K_5$, some two \emph{independent}
edges intersect, where independent means that the edges do not share a
vertex.
\end{fact}

This fact is a consequence of the Hanani--Tutte theorem 
mentioned above Proposition~\ref{p:simple-crpcr},
although that theorem is somewhat too big a hammer
for this purpose. But proving the fact
rigorously is harder than it may seem, even if we
assume nonplanarity of $K_5$ as known (although
a rigorous proof of the nonplanarity
is almost never included in  graph theory courses).

\begin{exercise}\label{ex:badK5pf} Find a mistake
in the following ``proof'' of Fact~\ref{f:K5i}:
Consider a (standard) drawing of $K_5$. If two independent
edges cross, we are done, and otherwise, some two edges
sharing a vertex cross. But such crossings can be removed
by the following transformation
\immfig{j-untangle-adj}
and so eventually we reach a plane drawing---a contradiction.
% may introduce crossings of independent edges
\end{exercise}

\begin{proofof}{Lemma~\ref{l:pcr-Kn}} By Fact~\ref{f:K5i},
in every drawing of $K_n$, every $5$-tuple 
of vertices induces  a pair of independent
edges that cross.
A given pair of independent crossing edges
determines $4$ vertices of the 5-tuple, and so the number
of 5-tuples inducing this particular pair of edges
is at most $n-4$. So $\PCR(K_n)\ge
{n\choose 5}/(n-4)=\Omega(n^4)$.
\end{proofof}

We remark that Lemma~\ref{l:pcr-Kn} also follows from a generally
useful result, the \defi{crossing lemma} of Ajtai et al.~\cite{acns-cfs-82} and
 Leighton \cite{Leighton84},
which asserts that every graph with $n$ vertices
and $m\ge 4n$ edges has crossing number $\Omega(m^3/n^2)$.
We actually need a version of the lemma for the pair-crossing
number, which holds with the same bound, as was observed in
Pach and T\'oth \cite[Thm.~3]{PachToth-cr}.\footnote{The argument
in their proof is not quite correct, but the problem is rectified in 
Remark~2 in Section~3 of~\cite{PachToth-cr}.} This proof does not avoid
Fact~\ref{f:K5i}---it actually relies on a generalization of it.

\begin{proofof}{Proposition~\ref{p:str-vcong}}
Let $G=(V,E)$,  
and let $(\gamma_v:v\in V)$ be a string representation of $G$.
We are going to produce a drawing of the complete graph
$K_V$ on the vertex set~$V$.

We draw each vertex $v\in V$ as a
point $p_v\in\gamma_v$, in such a way that
all the $p_v$ are distinct. 

For every edge
$\{u,v\}\in {V\choose 2}$ of the complete graph, we pick a path
$P_{uv}$ from $\PP_{uv}$, in a way to be specified later.
Let us enumerate the vertices along $P_{uv}$
as $v_0=u,v_1,v_2,\ldots,v_k=v$. Then we draw the edge $\{u,v\}$
of $K_V$ in the following manner: we start at $p_u$, follow
$\gamma_{u}$ until some (arbitrarily chosen) intersection with
$\gamma_{v_1}$, then we follow $\gamma_{v_1}$ until some intersection
with $\gamma_{v_2}$, etc., until we reach $\gamma_v$ and $p_v$ on it.
\immfig{j-along-path}
In this way we typically do not get a standard drawing, since
edges may share segments, have self-intersections and
triple points, and they may pass 
through vertices. However, we can obtain a standard drawing by 
shortcutting loops and a small
perturbation of the edges, in such a way that no new intersecting
pairs of edges are created.  Hence, by Lemma~\ref{l:pcr-Kn},
there are $\Omega(n^4)$ intersecting pairs of edges in
the original drawing as well. We are going to estimate the
number of intersecting pairs in a different way.

We note that the drawings of two edges $\{u,v\}$ and $\{u',v'\}$
cannot intersect unless there are vertices $w\in P_{uv}$ and 
$w'\in P_{u'v'}$ such that
$\gamma_w\cap \gamma_{w'}\ne \emptyset$,
i.e., $\{w,w'\}\in E(G)$ or $w=w'$.
Let us write this latter condition as $P_{uv}\sim P_{u'v'}$.

How do we select the paths $P_{uv}$? For this, we consider a flow
$\varphi$ for which $\vcong(G)$ is attained. Since there is a unit
flow between every pair of vertices $\{u,v\}$, the values
of $\varphi(P)$ define a probability distribution on $\PP_{uv}$.
We choose $P_{uv}\in\PP_{uv}$ from this distribution at random,
the choices independent
for different $\{u,v\}$.

The number $X$ of intersecting pairs of edges in this drawing
is a random variable, and we bound from above its expectation.
First we note that 
%\begin{eqnarray*}
\[
\Prob{\,\{u,v\} \mbox{ and }\{u',v'\}\mbox{ intersect}}
\le \Prob{P_{uv}\sim P_{u'v'}}\]
\[
=
\sum_{P\in\PP_{uv},P'\in \PP_{u'v'}, P\sim P'} \Prob{P_{uv}=P\mbox{ and }
P_{u'v'}=P'}
\]
\[
= \sum_{P\in\PP_{uv},P'\in \PP_{u'v'}, P\sim P'} 
\varphi(P)\varphi(P')
\]
%\end{eqnarray*}
(for the last equality we have used independence).
Therefore
\begin{eqnarray*}
\Ex{X} &=& \sum_{\{\{u,v\},\{u',v'\}\}\in{{V\choose 2}\choose 2}} 
\Prob{\,\{u,v\} \mbox{ and }\{u',v'\}\mbox{ intersect}}
\\
&\le& \sum_{\{u,v\},\{u',v'\}}\ \ \  \sum_{P\in\PP_{uv},P'\in \PP_{u'v'}, P\sim P'} 
\varphi(P)\varphi(P')
\\
&=&\sum_{\{w,w'\}\in E{\rm~or~} w=w'}\ \ \ 
\sum_{P,P'\in\PP, w\in P,w'\in P'}\varphi(P)\varphi(P')
\\
&=&\sum_{\{w,w'\}\in E{\rm~or~} w=w'}
\biggl(\sum_{P\in\PP,w\in P}\varphi(P)\biggr)
\biggl(\sum_{P'\in\PP,w'\in P}\varphi(P')\biggr).
\end{eqnarray*}
The first sum in parentheses is at most $2\vcong(w)$,
and the second one at most $2\vcong(w')$; the $2$
is needed because of the paths $P$ for which $w$ is an end-vertex.
The number of terms in the outer sum is $|E|+n\le 2m$.
Altogether we  get $\Ex{X}\le 8m \vcong(G)^2$.

Since, on the other hand, we always have $X=\Omega(n^4)$, we obtain
$1/\vcong(G)=O(\sqrt m/n^2)$ as claimed.
\end{proofof}

\section{Flows, cuts, and metrics: the edge case}\label{s:edual}

Here we prove the edge version of the approximate flow/cut
duality, Theorem~\ref{t:leightonrao}. We essentially follow
an argument of Linial, London, and Rabinovich \cite{llr-liegc-94}.

\heading{Dualizing the linear program. } The first step of the
proof can be concisely expressed as follows: express $\econg(G)$
by a linear program, dualize it, and see what the dual 
means.\footnote{We assume that the reader has heard about linear
programming and the duality theorem in it; if not, we recommend
consulting a suitable source. Here is a very brief summary:
A \defi{linear program} is the computational problem of maximizing
(or minimizing) a linear
function over the intersection of finitely many half-spaces in
$\R^n$ (i.e., a convex polyhedron). Every linear program can
be converted to a \emph{standard form}: $\max_{\xx\in P}\cc^T\xx$
with $P=\{\xx\in\R^n: A\xx\le\bb, \xx\ge \bzero\}$, where
 $\cc\in\R^n$, $\bb\in\R^m$, $A$ is an $m\times n$ matrix, and the inequalities between vectors
are meant componentwise. The \emph{dual} of this linear program
is $\min_{\yy\in D}\bb^T\yy$, $D=\{\yy\in\R^m, A^T\yy\ge\cc,\yy\ge\bzero\}$,
and the \emph{duality theorem of linear programming}
asserts that if $P\ne\emptyset\ne D$, then 
$\min_{\xx\in P}\cc^T\xx=\max_{\yy\in D}\bb^T\yy$.}

It is slightly nicer to work with $\frac1{\econg(G)}$, which can be
expressed as the maximum $t$ such that there is a flow $\varphi$
with edge congestion at most $1$ that sends at least $t$ between
every pair of vertices. The resulting linear program has variables
$t\ge 0$ and 
$\varphi(P)$, $P\in\PP$, and it looks like this:
\begin{eqnarray}
%\frac1{\econg(G)}=
\max \Big\{ t\ge 0\!\!\!&\!\!\!:\!\!\!& \varphi(P)\ge 0\mbox{ for all }P\in\PP,\nonumber\\
&& \textstyle \sum_{P:e\in P} \varphi(P)\le 1\mbox{ for all }e\in E,\label{e:cng}\\
&& \textstyle \sum_{P\in\PP_{uv}}\varphi(P)\ge t\mbox{ for all }\{u,v\}\in
{\textstyle {V\choose 2}}
\Big\}.\label{e:flw}
\end{eqnarray}
The variables of the dual linear program are $x_e$, $e\in E$,
corresponding to the constraints (\ref{e:cng}), and
$y_{uv}$, $\{u,v\}\in{V\choose 2}$, corresponding to
the constraints (\ref{e:flw}). The dual reads
\begin{eqnarray}
\min \Big\{\textstyle  \sum_{e\in E}x_e\!\!\!&\!\!\!:\!\!\!& x_e,y_{uv}\ge 0,\nonumber\\
&&\textstyle  \sum_{e\in P} x_e\ge y_{uv}\ifafour\mbox{ for every }\else,\ \fi
P\in\PP_{uv},~
\{u,v\}\in {\textstyle {V\choose 2}},\label{e:shopa}\\
&&\textstyle  \sum_{\{u,v\}\in{V\choose 2}}y_{uv}\ge 1
\Big\}, \label{e:tot}
\end{eqnarray}
and its value also equals $\frac1{\econg(G)}$ by the duality theorem.
(Checking this claim carefully takes some work, and we expect only the 
most diligent readers to verify it---the others may simply take it for granted,
since the linear programming duality is a side-topic for us.)

Fortunately, the dual linear program has a nice interpretation.
We think of the variables $x_e$ as edge weights, and then the constraints
(\ref{e:shopa}) say that $y_{uv}$ is at most the sum of weights along
every $u$-$v$ path. From this it is easy to see that in an optimal
solution of the dual linear program, each $y_{uv}$ is the length
of a shortest $u$-$v$ path under the edge weights given by the $x_e$.
When the $y_{uv}$ are given in this way, we may also assume that
for every edge $e=\{u,v\}\in E$, we have $x_e=y_{uv}$:
Indeed,  if $x_e>y_{uv}$, then there is a shortcut between $u$ and $v$
bypassing the edge $e$, i.e., a $u$-$v$ path of length $y_{uv}$.
So if we decrease $x_e$ to the value $y_{uv}$, the length
of a shortest path between every two vertices remains unchanged
and thus no inequality in the linear program is violated, while
$\sum_{e\in E}x_e$ decreases.

Thus, if we write $d_w$ for the shortest-path (pseudo)metric\footnote{We
recall that a \defi{metric} on a set $V$ is a mapping $d\:V\times V\to
[0,\infty)$ satisfying (i) $d(u,v)=d(v,u)$ for all $u,v$;
(ii) $d(u,u)=0$ for all $u$; (iii) $d(u,v)>0$ whenever $u\ne v$;
and (iv) $d(u,v)\le d(u,x)+d(x,v)$ for all $u,v,x\in V$ (triangle
inequality). A \emph{pseudometric} satisfies the same axioms
except possibly for (iii).}
induced on $V$ by an edge weight function $w\:E\to [0,\infty)$,
we can express the conclusion of the dualization step as
\begin{equation}\label{e:thatfrac}
\frac1{\econg(G)}=\min \Bigl\{
\frac{\sum_{\{u,v\}\in E}d_w(u,v)}{\sum_{\{u,v\}\in {V\choose 2}}d_w(u,v)}:
\ w\:E\to[0,\infty), w\not\equiv 0\Bigr\}.
\end{equation}
Here $w\not\equiv 0$ means that $w$ is not identically $0$; note that
we replaced the constraint (\ref{e:tot}), requiring the sum
of all distances under $d_w$ to be at least~$1$, by dividing
the minimized function by the sum of all distances.

\heading{Cut metrics and line metrics. } To make further progress,
we will investigate the minimum of the same ratio as in
\eqref{e:thatfrac}, but over different classes of metrics.

A \defi{cut metric} on a set $V$ is a pseudometric\footnote{Cut
\emph{metric} is really a misnomer, since a cut metric is
almost never a metric; we should speak of a
\emph{cut pseudometric}, but we conform to the usage in the literature.
A similar remark applies to line metrics considered below.}
 $c$ given by $c(u,v)=|f(u)-f(v)|$ for some function $f\:V\to \{0,1\}$.

By comparing the definitions, we can express the edge sparsity
of a graph as
\begin{equation}\label{e:min-cutm}
\espars(G)= \min \Bigl\{
\frac{\sum_{\{u,v\}\in E}c(u,v)}{\sum_{\{u,v\}\in {V\choose 2}}c(u,v)}:
c\mbox{ a cut metric on }V, c\not\equiv 0\Bigr\}
\end{equation}
(please check). 

Next, it turns out that we can replace cut metrics by line metrics
in \eqref{e:min-cutm} and the minimum stays the same. Here
a \defi{line metric} is a pseudometric $\ell$ such that
$\ell(u,v)=|f(u)-f(v)|$ for some function $f\:V\to\R$.
We leave the proof as an instructive exercise.

\begin{exercise} Show that the minimum in
\begin{equation}\label{e:linmetr}
\min \Bigl\{
\frac{\sum_{\{u,v\}\in E}\ell(u,v)}{\sum_{\{u,v\}\in {V\choose 2}}\ell(u,v)}:
\ell\mbox{ a line metric on }V, \ell\not\equiv 0\Bigr\}
\end{equation} 
is attained by a cut metric, and hence it also equals $\espars(G)$.
(Hint: show that if the function $f$ defining a line metric $\ell$
attains at least three distinct values, then some value can be eliminated.)
\end{exercise}

A key result that allows us to compare the minimum \eqref{e:thatfrac}
over all shortest-path metrics
with the minimum \eqref{e:linmetr} over all line metric follows from
the work of Bourgain \cite{Bourgain-l2}. His main theorem was
formulated differently, but his proof immediately yields the
following formulation, which is the most convenient for our purposes.

\begin{theorem}\label{t:bou-a}
 Let $V$ be an $n$-point set. For every (pseudo)metric
$d$ on $V$ there exists a line metric $\ell$ on $V$ satisfying
\begin{enumerate}\itemsep=0pt\topsep=0pt
\item[\rm(i)] ($\ell$ is below $d$) $\ell(u,v)\le d(u,v)$ for all $u,v\in V$, and
\item[\rm(ii)] (the average distance not decreased too much)
\[
{\textstyle\sum_{\{u,v\}\in{V\choose 2}}\ell(u,v)} \ge \frac c{\log n}
{\textstyle\sum_{\{u,v\}\in{V\choose 2}} d(u,v)},
\]
for a constant $c>0$.
\end{enumerate} 
\end{theorem}

For completeness, we demonstrate the main idea of the proof 
in Section~\ref{s:bourg} below.

\begin{proofof}{Theorem~\ref{t:leightonrao}}
Let $d^*$ be a shortest-path metric attaining the minimum
in the expression \eqref{e:thatfrac} for $\frac1{\econg(G)}$.
We apply Theorem~\ref{t:bou-a} with $d=d^*$ and obtain a line
metric $\ell^*$ satisfying (i), (ii) in the theorem. Then
\begin{eqnarray*}
\frac1{\econg(G)}&=&\frac{\sum_{\{u,v\}\in E}d^*(u,v)}{\sum_{\{u,v\}\in {V\choose 2}}d^*(u,v)}\ge 
\frac c{\log n} \cdot\frac{\sum_{\{u,v\}\in E}\ell^*(u,v)}{\sum_{\{u,v\}\in {V\choose 2}}\ell^*(u,v)}\\
&\ge& \frac c{\log n} \cdot\min \Bigl\{
\frac{\sum_{\{u,v\}\in E}\ell(u,v)}{\sum_{\{u,v\}\in {V\choose 2}}\ell(u,v)}:
\ell\mbox{ a line metric on }V, \ell\not\equiv 0\Bigr\}\\
&=& \frac c{\log n} \cdot \espars(G).
\end{eqnarray*}
\end{proofof}

\section{Proof of a weaker version of Bourgain's theorem}\label{s:bourg}

Here we prove a version of Theorem~\ref{t:bou-a} with $\log n$
replaced by $\log^2 n$; this weakening makes the proof simpler, while
preserving the main ideas.

Providing a line metric $\ell$ satisfying condition (i),
$\ell\le d$, is equivalent to providing a function $f\:V\to\R$
that is \defi{$1$-Lipschitz}, i.e., satisfies $|f(u)-f(v)|\le
d(u,v)$ for all $u,v\in V$.

A suitable $f$ is chosen at random, in the following steps.
\begin{enumerate}
\item Let $k$ be the smallest integer with $2^k\ge n$,
i.e., $k=\lceil\log_2 n\rceil$.
Choose an index $j\in\{0,1,\ldots,k\}$
uniformly at random, and set $p:=2^{-j}$.
\item Choose a random subset $A\subseteq V$, where each point
$v\in V$ is included in $A$ independently with probability~$p$.
\item Define $f$ by $f(u):=d(u,A)=\min_{a\in A} d(u,a)$.
\end{enumerate}

A nice thing about this way of choosing $f$ is that it is $1$-Lipschitz
for every $A\subseteq V$, as can be easily checked using the triangle 
inequality. So it remains to show that, with positive probability,
the line metric induced by $f$ satisfies a weaker version of condition (ii),
i.e., that it does not decrease the average distance too much.

We will actually prove that for every $u,v\in V$, $u\ne v$,
\begin{equation}\label{e:someprob}
\Prob{|f(u)-f(v)|\ge\textstyle \frac{c_0}{\log n}\cdot d(u,v)    }\ge
\frac{c_0}{\log n},
\end{equation}
where the probability is with respect to the random choice of $f$ as above,
and $c_0>0$ is a suitable constant.
Assuming \eqref{e:someprob}, passing to expectation,
 and summing over $\{u,v\}\in{V\choose 2}$,
we arrive at
\[
\Ex{\textstyle\sum_{\{u,v\}\in{V\choose 2}}|f(u)-f(v)|}\ge
\frac{c_0^2}{\log^2 n}\textstyle\sum_{\{u,v\}\in{V\choose 2}}d(u,v),
\]
and hence at least one $f$ satisfies (ii) with $\log^2 n$ instead of $\log n$.

So we fix $u,v$ and we aim at proving \eqref{e:someprob}. Let us set
$\Delta:=d(u,v)/(2k-1)$. We have $|f(u)-f(v)|=|d(u,A)-d(v,A)|$,
and the latter expression is at least $\Delta$ provided that,
for some $r\ge 0$, the set $A$ intersects the (closed) $r$-ball around $u$
and avoids the (open) $(r+\Delta)$ ball around $v$, or the other way round.
\immfig{j-rdelta}
In order for this event to have a non-negligible probability, we need
that the number of points in the bigger balls is not much
larger than in the smaller ball.
The trick for achieving this is to consider a system of balls
as in the next picture:
\immfig{j-boub}
The picture is for $k=4$. In general, $B_i$ is the closed
ball of radius $i\Delta$, $i=0,1,\ldots,k$, centered at $u$ for even $i$
and at $v$ for odd $i$. Let $B_i^\circ$ denote the corresponding open
ball (all points at distance strictly smaller than $i\Delta$
from the center).

Let $n_i$ be the number of points in $B_i$. We claim that
$n_{i+1}/n_i\le 2$ for some $i\in\{0,1,\ldots,k-1\}$; indeed, if not, then
$|B_k|>2^k\ge n$---a contradiction.

We fix such an $i$, and we also fix $j_0$ such that $n_i$
is approximately $2^{j_0}$; more precisely,
$2^{j_0}\le n_i<2^{j_0+1}$. 

Let $p=2^{-j_0}$, and let us pick a random $A$ as in the second
step of the choice of $f$ with this value of $p$. 
By a simple calculation, which we leave as an exercise, there is a 
constant $c_1>0$ such that
\[
\Prob{A\cap B_i\ne\emptyset\mbox{ and } A\cap B_{i+1}^\circ=\emptyset}\ge c_1.
\]

\begin{exercise}
Let $X,Y$ be disjoint sets, and let $A\subseteq X\cup Y$ be a random
subset, where each point of $X\cup Y$ is included in $A$ with probability $p$,
independent of all other points, with
$0<p\le\frac12$.
Assuming $\frac 1{2p}\le |X|,|Y| \le \frac 2p$, show that
$\Prob{A\cap X\ne\emptyset\mbox{ and } A\cap Y=\emptyset}\ge c_1$
for a constant $c_1>0$.
\end{exercise}

Now \eqref{e:someprob} follows easily: given $u,v$, the probability
of choosing $j=j_0$ is $\frac1{k+1}=\Omega(\frac1{\log n})$,
and conditioned on this choice, we have $\Prob{|f(u)-f(v)|\ge
\Delta}\ge c_1$. This concludes the proof of the weaker
version of Theorem~\ref{t:bou-a}.

\section{Flows, cuts, and metrics: the vertex case}\label{s:vdual}

Here we prove Theorem~\ref{t:vleightonrao}, the vertex case of the
approximate duality, and this will also conclude the quest for the
proof of the separator theorem for string graphs. Initially
 we proceed in a way similar to the edge case from 
Section~\ref{s:edual}, but the last step is more demanding and uses a nice 
method for producing sparse vertex cuts algorithmically.

\heading{Dualization again. } As before, we write $\frac1{\vcong(G)}$
as a linear program and dualize it. The linear program differs
from the one for $\frac1{\econg(G)}$ only in the second line:
\begin{eqnarray}
\max \Big\{ t\ge 0\!\!\!&\!\!\!:\!\!\!& \varphi(P)\ge 0\mbox{ for all }P\in\PP,\nonumber\\
&& \textstyle \sum_{P:v\in P} \grayhalf\varphi(P)\le 1\mbox{ for all }v\in V,\label{e:cngv}\\
&& \textstyle \sum_{P\in\PP_{uv}}\varphi(P)\ge t\mbox{ for all }\{u,v\}\in
{\textstyle {V\choose 2}}
\Big\}\label{e:flwv};
\end{eqnarray}
here the meaning of $\grayhalf$ is as in the definition of 
$\vcong(G)$ in Section~\ref{s:approxdual}.
In the dual, we have variables $y_{uv}$ indexed by pairs of vertices
and $z_v$ indexed by vertices, and it reads
\begin{eqnarray}
\min \Big\{\textstyle  \sum_{z\in V}z_v\!\!\!&\!\!\!:\!\!\!& z_v,y_{uv}\ge 0,\nonumber\\
&&\textstyle  \sum_{v\in P} \grayhalf z_v\ge y_{uv}
\ifafour\mbox{ for every }\else,\ \fi
P\in\PP_{uv},~
\{u,v\}\in {\textstyle {V\choose 2}},\label{e:shopav}\\
&&\textstyle  \sum_{\{u,v\}\in{V\choose 2}}y_{uv}\ge 1
\Big\}. \label{e:totv}
\end{eqnarray}
This, too, can be interpreted using a metric on $G$. This
time we have a function $s\:V\to[0,\infty)$ assigning weights to
vertices. Let us define the derived weight of an edge $e=\{u,v\}$
by $w(e):=\frac12(s(u)+s(v))$ and denote the corresponding shortest-path
metric by $d_s$. Then, in a way very similar to the edge case, one
can see that
\begin{equation}\label{e:thatfracv}
\frac1{\vcong(G)}=\min \Bigl\{
\frac{\sum_{v\in V}s(v)}{\sum_{\{u,v\}\in {V\choose 2}}d_s(u,v)}:
\ s\:V\to[0,\infty), s\not\equiv 0\Bigr\}.
\end{equation}
Here the convention with $\grayhalf$ for the vertex congestion pays off---the
dual has a nice interpretation in terms of shortest-path metrics.

Let $s^*$ be a weight function for which the minimum in \eqref{e:thatfracv}
is attained. Applying Bourgain's theorem (Theorem~\ref{t:bou-a})
to the metric $d_{s^*}$ yields a function $f^*\:V\to\R$ 
that is $1$-Lipschitz w.r.t.\ $d_{s^*}$ and satisfies
\[
\frac{\sum_{v\in V}s(v)}{\sum_{\{u,v\}\in {V\choose 2}}|f^*(u)-f^*(v)|}
=O\left(\frac{\log n}{\vcong(G)}\right).
\]

The following theorem of
Feige, Hajiaghayi, and Lee~\cite{feige2008improved}
then shows how such an $f^*$ can be used to produce sparse vertex
cuts in~$G$. This is the last step in the proof of 
Theorem~\ref{t:vleightonrao}.

\begin{theorem}\label{t:fhl}
Let $G$ be a graph, $s\:V\to[0,\infty)$ a weight function on
the vertices, $d_s$ the corresponding metric, and let $f\:V\to\R$
be a non-constant $1$-Lipschitz function w.r.t.\ $d_s$. Then
\[
\vspars(G)\le \frac{\sum_{v\in V}s(v)}
{\sum_{\{u,v\}\in {V\choose 2}}|f(u)-f(v)|}.
\]
\end{theorem}

\begin{proof} The proof actually provides a polynomial-time algorithm
for finding a vertex cut with sparsity bounded as in the theorem.

Let us number the vertices of $G$ so that $f(v_1)\le f(v_2)\le\cdots
\le f(v_n)$. For every $i=1,2,\ldots,n-1$, we are going to find
a vertex cut $(A_i,B_i,S_i)$, and show that one of these will do.

To this end, given $i$, we form an auxiliary graph $G_i^+$ by
adding new vertices $x$ and $y$ to $G$, connecting $x$ to
$v_1$ through $v_i$, and $y$ to $v_{i+1}$ through $v_n$.
\immfig{j-Giplus}
We let $S_i\subseteq V$ be a minimum cut in $G_i^+$ separating 
$x$ from $y$ (which can be found using a max-flow algorithm, for example).
Let $A_i:=\{v_1,\ldots,v_i\}\setminus S_i$ and 
$B_i:=\{v_{i+1},\ldots,v_n\}\setminus S_i$,
and let
\[
\alpha:=\min_i \vspars(A_i,B_i,S_i)=\min_i \frac{|S_i|}{|A_i\cup S_i|\cdot
|B_i\cup S_i|}.
\]
Since $\{v_1,\ldots,v_i\}\subseteq A_i\cup S_i$, we have
$|A_i\cup S_i|\ge i$, and similarly $|B_i\cup S_i|\ge n-i$. Thus, for every
$i$ we have
\begin{equation}\label{e:in-i}
|S_i|\ge \alpha i(n-i).
\end{equation}

In order to prove the theorem, we want to derive
\begin{equation}\label{e:sum-uv}
\textstyle\alpha\sum_{\{u,v\}\in {V\choose 2}}|f(u)-f(v)|\le \sum_{v\in V}s(v).
\end{equation}
Setting $\eps_i=f(v_{i+1})-f(v_i)\ge 0$, we can rearrange the left-hand side:
$\alpha\sum_{i<j} (f(v_j)-f(v_i))=\alpha\sum_{i=1}^{n-1}i(n-i)\eps_i$
(we just look how many times the segment between $f(v_{i})$ and
$f(v_{i+1})$ is counted). Then, substituting from \eqref{e:in-i},
we finally bound the left-hand side of \eqref{e:sum-uv}
by $\sum_{i=1}^{n-1}\eps_i|S_i|$. It remains to prove
\begin{equation}\label{e:333}
\textstyle \sum_{i=1}^{n-1}\eps_i|S_i|\le \sum_{v\in V}s(v),
\end{equation}
and this is the most
ingenious part of the proof.

Roughly speaking, for every term $\eps_i|S_i|$, we want
to find vertices of sufficient total weight sufficiently
close to the interval $[f(v_i),f(v_{i+1})]$. We use Menger's
theorem, which guarantees that there are $|S_i|$ vertex-disjoint
paths that have to ``jump over'' the interval $[f(v_i),f(v_{i+1})]$.

More precisely, we express both sides of \eqref{e:333}
as integrals. Namely, we write $\sum_{i=1}^{n-1}\eps_i|S_i|
=\int_{-\infty}^\infty g(z)\dd z$, where $g$ is the function
that equals $|S_i|$ on $[f(v_i),f(v_{i+1}))$ and $0$ elsewhere:
\immfig{j-funcg}
Similarly, $\sum_{v\in V}s(v)=\int_{-\infty}^\infty \sum_{i=1}^n h_i(z)\dd z$,
where $h_i$ is the function equal to $1$ on $[f(v_i)-\frac{s(v_i)}2,
f(v_i)+\frac{s(v_i)}2]$ and to $0$ elsewhere:
\immfig{j-funchi}
We claim that $g(z)\le \sum_{i=1}^n h_i(z)$ for every $z\in\R$;
this will imply \eqref{e:333}. 

Let $z\in [f(v_i),f(v_{i+1})]$, and set $m=g(z)=|S_i|$. We want
to show $\sum_{i=1}^n h_i(z)\ge m$, which means that we need to find
$m$ distinct vertices $v$ such that $|f(v)-z|\le\frac{s(v)}2$;
let us call such $v$ the \emph{paying vertices} since we can
imagine that they pay for~$g(z)$.

As announced, we use Menger's theorem, which tells us that, since
$S_i$ is a minimum $x$-$y$ cut in $G_i^+$, there are $m$
$x$-$y$ paths $P_1,\ldots,P_m$ that are vertex-disjoint except
for sharing the end-vertices $x$ and $y$. Each $P_j$ contains at
least one edge $e_j=\{a_j,b_j\}$ with one endpoint among
$v_1,\ldots,v_i$ and the other among $v_{i+1},\ldots,v_n$.
Hence $z\in [f(a_j),f(b_j)]$, and since $f$ is $1$-Lipschitz,
$|f(a_j)-f(b_j)|\le d_s(a_j,b_j)
\le \frac12(s(a_j)+s(b_j))$. Thus, we have $|f(a_j)-z|\le \frac{s(a_j)}2$
or $|f(b_j)-z|\le \frac{s(b_j)}2$ (or both), and so $a_j$ or $b_j$
is a paying vertex. This gives the desired $m$ distinct paying
vertices.
\end{proof}

\localbib

\end{document}